\documentclass[10pt]{amsart}
\usepackage{amsmath,amsfonts,amscd,amssymb,graphicx,mathrsfs,eufrak}
\usepackage[dvips,all,arc,curve,color,frame]{xy}
\usepackage[usenames]{color}
\usepackage[colorlinks]{hyperref}

\usepackage{tikz,mathrsfs}
\usetikzlibrary{arrows,decorations.pathmorphing,decorations.pathreplacing,positioning,shapes.geometric,shapes.misc,decorations.markings,decorations.fractals,calc,patterns}

\tikzset{>=stealth',
         cvertex/.style={circle,draw=black,inner sep=1pt,outer sep=3pt},
         vertex/.style={circle,fill=black,inner sep=1pt,outer sep=3pt},
         star/.style={circle,fill=yellow,inner sep=0.75pt,outer sep=0.75pt},
         tvertex/.style={inner sep=1pt,font=\scriptsize},
         gap/.style={inner sep=0.5pt,fill=white}}

\newtheorem{thm}{Theorem}[section]
\newtheorem{prop}[thm]{Proposition}
\newtheorem{lemma}[thm]{Lemma}
\newtheorem{defin}[thm]{Definition}
\newtheorem{cor}[thm]{Corollary}

\theoremstyle{definition} 
 
\newtheorem{example}[thm]{Example}

\newtheorem{remark}[thm]{Remark}

\newcommand{\dsg}[1]{\mathfrak{D}_{\mathbf{sg}}(#1)}

\newcommand{\stq}[1]{\underline{\Mod{#1}}}
\newcommand{\dsgaf}[1]{D^b(\mod{#1})/ K^b(\proj{#1})}

\newcommand{\Sstq}[1]{\tt{St}(\stq{R})}

\newcommand{\C}[1]{\mathbb{C}^{#1}}
\newcommand{\m}{\mathfrak{m}}

\newcommand{\s}[1]{\mathscr{#1}}
\renewcommand{\c}[1]{\mathcal{#1}}
\renewcommand{\u}[1]{\underline{#1}}

\newcommand{\w}[1]{\widetilde{#1}}

\renewcommand{\t}[1]{\textnormal{#1}}
\renewcommand{\tt}[1]{\mathtt{#1}}

\def\GL{\mathop{\rm GL}\nolimits}
\def\CM{\mathop{\rm CM}\nolimits}
\def\SCM{\mathop{\rm SCM}\nolimits}
\def\OCM{\mathop{\Omega{\rm CM}}\nolimits}

\def\mod{\mathop{\rm mod}\nolimits}

\def\coh{\mathop{\rm coh}\nolimits}
\def\Mod{\mathop{\rm Mod}\nolimits}

\def\proj{\mathop{\rm proj}\nolimits}

\def\pd{\mathop{\rm proj.dim}\nolimits}

\def\Hom{\mathop{\rm Hom}\nolimits}
\def\End{\mathop{\rm End}\nolimits}
\def\Ext{\mathop{\rm Ext}\nolimits}

\def\Tr{\mathop{\rm Tr}\nolimits}
\def\add{\mathop{\rm add}\nolimits}

\def\Spec{\mathop{\rm Spec}\nolimits}

\def\gl{\mathop{\rm gl.dim}\nolimits}
\def\AA{\mathop{\mathcal{A}}\nolimits}
\def\CC{\mathop{\mathcal{C}}\nolimits}

\def\Db{\mathop{\rm{D}^b}\nolimits}

\def\AA{\mathop{\mathcal{A}^{}}\nolimits}
\def\BB{\mathop{\mathcal{B}^{}}\nolimits}
\def\CC{\mathop{\mathcal{C}^{}}\nolimits}

\def\TT{\mathop{\mathcal{T}^{}}\nolimits}

\makeatletter
\@addtoreset{equation}{section}

\makeatother

\CompileMatrices
\MakeOutlines

\begin{document}
\title{\textsc{A New Triangulated Category for Rational Surface Singularities}}
\author{Osamu Iyama}
\address{Osamu Iyama\\ Graduate School of Mathematics\\ Nagoya University\\ Chikusa-ku, Nagoya, 464-8602, Japan}
\email{iyama@math.nagoya-u.ac.jp}
\author{Michael Wemyss}
\address{Michael Wemyss\\ Graduate School of Mathematics\\ Nagoya University\\ Chikusa-ku, Nagoya, 464-8602, Japan}
\thanks{The second author was supported by a JSPS Postdoctoral Fellowship.}
\email{wemyss.m@googlemail.com}
\begin{abstract}
In this paper we introduce a new triangulated category for rational surface singularities which in the non-Gorenstein case acts as a substitute for the stable category of matrix factorizations.  The category is formed as a stable quotient of the Frobenius category of special CM modules, and we classify the relatively projective-injective objects and thus describe the AR quiver of the quotient.  Connections to the corresponding reconstruction algebras are also discussed.
\end{abstract}
\maketitle
\parindent 20pt
\parskip 0pt

\section{Introduction}
The theory of almost split sequences first entered the world of quotient singularities through the work of Auslander \cite{Auslander_rational}.  Rather than interpreting the McKay correspondence for finite subgroups $G\leq \t{SL}(2,\C{})$ in terms of representations of $G$, he instead viewed the representations as Cohen--Macaulay (=CM) $\C{}[[x,y]]^G$-modules and showed that the Auslander--Reiten (=AR) quiver coincides with the McKay quiver, thus linking with the geometry through the dual graph of the minimal resolution.

There is a benefit to this viewpoint, since considering representations as modules we may sum them together (without multiplicity) and consider their endomorphism ring; this is Morita equivalent to the skew group ring $\C{}[x,y]\# G$.  Through projectivization (\cite{A},\cite{ARS}) the theory of almost split sequences can be used to gain homological insight into the structure of the endomorphism ring, and furthermore it can be used to recover the relations on the McKay quiver which yields a presentation of the algebra \cite{RV}.

Recently \cite{Wemyss_GL2} it was realized that for quotients by groups not inside $\t{SL}(2,\C{})$ the skew group ring is far too large, and instead we should sum less CM modules together and consider this endomorphism ring instead.  The modules that we sum are the special CM modules, and the resulting endomorphism ring is called a reconstruction algebra.  These algebras are in fact defined for all rational surface singularities (not just quotients), are always derived equivalent to the minimal resolution \cite[\S 2 and Lemma 3.2]{Wemyss_GL2} and have global dimension 2 or 3 \cite[2.10]{Iyama_Wemyss_specials}.  However the main difference between this new situation and the classical case is that the reconstruction algebra is very non-symmetrical and so for example writing down the relations is a much more delicate and difficult task.

We are thus motivated to study $\SCM(R)$ (or dually $\OCM(R)$), the category of special CM modules (respectively first syzygies of CM modules), from the viewpoint of relative AR theory \cite{AS} to try and gain an insight into this problem.  This short paper is dedicated to its study, and other related issues.

In this paper we show that $\SCM(R)$ admits a Frobenius structure and prove that the indecomposable relatively projective objects in $\SCM(R)$ are precisely $R$ together with those special CM modules which correspond to non-$(-2)$ curves in the dual graph of the minimal resolution.  Geometrically this means that the quotient $\u{\u{\SCM}}(R)$ only `sees' the crepant divisors.  Note that it is certainly possible that all special CM modules are relatively projective, in which case the quotient category is zero.  However there is still enough information to prove some results, for example that at all vertices in a reconstruction algebra corresponding to a $(-2)$ curve in the minimal resolution, there is only one relation which is a cycle at that vertex and further it is (locally) a preprojective relation. 

We remark that our triangulated category is a more manageable version of the rather large triangulated category of singularities $\dsg{R}=\dsgaf{R}$ which is well known in the Gorenstein case to coincide with $\underline{\CM}(R)$ \cite{Buch1}.  Note that our category is definitely not equivalent to $\dsg{R}$ since the category $\dsg{R}$ is always non-zero if $R$ is singular, but $\u{\u{\SCM}}(R)$ is zero if there are no crepant divisors in the minimal resolution. Furthermore our category $\u{\u{\SCM}}(R)$ is always Krull-Schmidt, a property not enjoyed by $\dsg{R}$ in the case when $R$ is not Gorenstein. It would be interesting to see if there are indeed any connections between the two categories.

\section{Conventions and Background}

Throughout this paper we let $R$ be a complete local normal domain of dimension two over an algebraically closed field of characteristic zero which furthermore is a rational singularity.  We denote $\CM(R)$ to be the category of Cohen-Macaulay (=CM) $R$-modules and by $\OCM(R)$ the category of first syzygies of CM $R$-modules. There is a duality
\[
(-)^*:=\Hom_R(-,R):\CM(R)\to \CM(R).
\]
Both $\CM(R)$ and $\OCM(R)$ are Krull--Schmidt categories since $R$ is complete.

We always denote the minimal resolution of $\Spec R$ by $\pi:\w{X}\rightarrow \Spec R$ and the irreducible exceptional curves by $\{ E_i\}_{i\in I}$.  For a sheaf $\s{F}$ on $\w{X}$, we denote $\s{F}^{\vee}$ to be the sheaf $\s{H}om_{\w{X}}(\s{F},\s{O}_{\w{X}})$
and we denote $\mathbf{T}(\s{F})$ to be the torsion subsheaf of
$\s{F}$, i.e. the kernel of the natural map $\s{F}\to\s{F}^{\vee\vee}$.  The following important lemma--definition is due to Esnault.
\begin{lemma}[{\cite[2.2]{Esnault_full}}]
Let $\s{F}$ be a sheaf on $\w{X}$.  There exists a CM $R$-module $M$ such that $\s{F}\cong\pi^\ast M/\mathbf{T}(\pi^\ast M)$ if and only if the following conditions are satisfied. \\
\t{(i)} $\s{F}$ is locally free.\\
\t{(ii)} $\s{F}$ is generated by global sections.\\
\t{(iii)} $H^1(\w{X},\s{F}^\vee\otimes \omega_{\w{X}})=0$
where $\omega_{\w{X}}$ is the canonical sheaf.\\
In this case $\s{F}\cong \pi^\ast M/\mathbf{T}(\pi^\ast M)$ is called a \emph{full sheaf}.  Moreover, for a full sheaf $\s{F}\cong \pi^\ast M/\mathbf{T}(\pi^\ast M)$ one has $\pi_\ast(\s{F})\cong M$ and $\pi_\ast( \s{F}^\vee)\cong M^*$.
\end{lemma}
To ease notation, for any CM module $M$ of $R$ we denote $\c{M}:=\pi^*M/\mathbf{T}(\pi^\ast M)$ to be the corresponding full sheaf on $\w{X}$.  In the following theorem--definition, if $M\in\mod R$ we denote ${\bf T}(M)$ to be the torsion submodule of $M$, i.e. the kernel of the natural map $M\to M^{**}$. 
\begin{thm}\label{characterization of SCM}
For $M\in\CM(R)$, the following conditions are equivalent.
\begin{itemize}
\item[(1)] $\Ext^1_{\w{X}}(\c{M},\s{O})=0$,
\item[(2)] $(M\otimes_R\omega_R)/{\bf T}(M\otimes_R\omega_R)\in\CM(R)$,
\item[(3)] $\Ext^2_R(\Tr M,\omega_R)=0$,
\item[(4)] $\Omega\Tr M\in\CM(R)$,
\item[(5)] $\Ext^1_R(M,R)=0$,
\item[(6)] $M^*\in\OCM(R)$,
\item[(7)] $\Omega M\cong M^*$ up to free summands.
\end{itemize}
We call such a module $M$ a \emph{special CM module}.  We denote $\SCM(R)$ to be the category of special CM modules.
\end{thm}
\begin{proof}
(1)$\iff$(2) is due to Wunram \cite{Wunram_generalpaper}, the remainder can be found in \cite[2.7, 3.5]{Iyama_Wemyss_specials}.
\end{proof}
The category $\SCM(R)$ is Krull--Schmidt since it is a full subcategory of $\CM(R)$, which is Krull--Schmidt.  Note that Theorem~\ref{characterization of SCM} implies that there is a duality
\[
(-)^*:\SCM(R)\to\OCM(R)
\]
and so these categories are intimately related. Some of our arguments in this paper rely on geometric notions and results, which we now recall :
\begin{defin}[\cite{Art}]
For exceptional curves $\{ E_i \}_{i\in I}$, define the fundamental cycle $Z_f=\sum_{i\in I}r_i E_i$ (with each $r_i\geq 1$) to be the unique smallest element such that $Z_f\cdot E_i\leq 0$ for all $i\in I$.
\end{defin}
We remark that $\s{O}_{Z_f}=\s{O}_{\w{X}}/\m\s{O}_{\w{X}}$ where $\m$ is the unique maximal ideal of $R$.
\begin{thm}[{\cite[1.2]{Wunram_generalpaper}}]\label{Wurnam_main_result}
\t{(a)} For every irreducible curve $E_i$ in the
exceptional divisor of the minimal resolution there is exactly one
indecomposable CM module $M_i$ (up to isomorphism) with
\[
H^1(\c{M}_i^\vee)=0
\]
and
\[
c_1(\c{M}_i)\cdot E_j=\delta_{ij} \mbox{ for all } i,j\in I.
\]
The rank of $M_i$ equals $r_i=c_1(\c{M}_i)\cdot Z_f$ where $Z_f=\sum r_i E_i$ is the fundamental cycle.\\
\t{(b)} Let $N\in\CM(R)$ be non-free and indecomposable.  Then there exists $i\in I$ such that $N\cong M_i$ if and only if $H^1({\c{N}}^\vee)=0$.
\end{thm}
Thus $\SCM(R)$ (and dually $\OCM(R)$) has only a finite number of indecomposable objects (i.e.\ has \emph{finite type}) since the non-free indecomposable special CM modules are in one-to-one correspondence with the exceptional curves in the minimal resolution of $\Spec R$.  It is known that $\CM(R)$ has finite type if and only if $R$ is a quotient singularity \cite{Auslander_rational}, so by passing to the special CM modules we can use finite-type algebra in a much broader setting.

Recall that for an additive category $\c{C}$ and an object $M\in \c{C}$, we denote $\add M$ to be the full subcategory of $\c{C}$ consisting of summands of finite direct sums of copies of $M$. We say that $M$ is an \emph{additive generator} of $\c{C}$ if $\c{C}=\add M$.  Thus $\SCM(R)$ has an additive generator $R\oplus\bigoplus_{i\in I}M_i$.
One further fact we will use is that its endomorphism algebra is derived
equivalent to the minimal resolution:

\begin{thm}\label{derivedequiv}
There is an equivalence of triangulated categories
\[
\Db(\coh\w{X})\approx\Db(\mod\End_R(R\oplus\bigoplus_{i\in I}M_i)).
\]
where $\{ M_i\}_{i\in I}$ denotes the set of non-free indecomposable special CM modules up to isomorphism.  We call $\End_R(R\oplus\bigoplus_{i\in I}M_i)$ the \emph{reconstruction algebra}.
\end{thm}
\begin{proof}
This follows by combining the discussion in \cite[\S 2]{Wemyss_GL2} (based entirely on \cite{VdB}) together with \cite[Lemma 3.2]{Wemyss_GL2}.
\end{proof}
Since $\w{X}$ is smooth it follows that the reconstruction algebra has finite global dimension, but even although $\dim\w{X}=2$ it is usually the case that the reconstruction algebra has global dimension three:

\begin{thm}
\[
\gl \End_R(R\oplus\bigoplus_{i\in I}M_i)=\left\{\begin{array}{cl} 2& \mbox{if $R$ is Gorenstein}\\ 3 & \mbox{else.}\end{array}\right. 
\]
\end{thm}
\begin{proof}
This is shown in both \cite[2.10]{Iyama_Wemyss_specials} and \cite[Corollary 3.3]{Wemyss_GL2}, but for the convenience of the reader here we give a different, more direct proof.

Set $M:=R\oplus\bigoplus_{i\in I}M_i$ and $\Lambda:=\End_R(M)$.  To show that $\gl \Lambda\leq 3$, by \cite[2.11]{Iyama_Wemyss_specials} we just need to show that for all $X\in\CM(R)$ there exists an exact sequence
\[
0\to M_1\to M_0\to X\to 0
\]
with each $M_i\in\add M$ such that
\[
0\to\Hom_R(M,M_1)\to \Hom_R(M,M_0)\to \Hom_R(M,X)\to 0
\]
is exact.  To do this, consider the extension 
\[
0\to R^a\to T\to X\to 0
\] 
corresponding to the minimal number of generators of $\Ext^1_R(X,R)$.  Applying $\Hom_R(-,R)$, we have an exact sequence
\[
\Hom_R(R^a,R)\to \Ext^1_R(X,R)\to \Ext^1_R(T,R)\to \Ext^1_R(R^a,R)=0
\]
where the left map is surjective from our choice of the
extension. Thus we have $\Ext^1_R(T,R)=0$ and so $T\in\SCM(R)=\add M$.  Further
\[
0\to\Hom_R(M,R^a)\to \Hom_R(M,T)\to \Hom_R(M,X)\to \Ext^1_R(M,R^a)=0
\]
since $M\in\SCM(R)$.  Hence $\gl \Lambda\leq 3$.  We have $\gl \Lambda\geq 2$ by the depth lemma. Furthermore $\gl \Lambda=2$ if and only if $\add M=\CM(R)$ (again by \cite[2.11]{Iyama_Wemyss_specials}), and $\add M=\CM(R)$ if and only if $R$ is Gorenstein.
\end{proof}

To summarize and fix notation, $\{ E_i\}_{i\in I}$ denotes the irreducible  exceptional curves in the minimal resolution of $\Spec R$.  We denote the non-free indecomposable special CM modules by $\{ M_i\}_{i\in I}$, where $M_i$ corresponds to the curve $E_i$.  Thus $M:=R\oplus\bigoplus_{i\in I}M_i$ is an additive generator of $\SCM(R)$.  The corresponding full sheaves on the minimal resolution will be denoted by $\{\c{M}_i\}_{i\in I}$.  For any CM $R$-module $N$, we often denote the chern class $c_1(\c{N})$ by simply $c_1(N)$.

\section{Syzygies and Chern Classes}
When $R$ is Gorenstein every CM module is special and so the categories $\OCM(R)$ and $\SCM(R)$ coincide; they both equal $\CM(R)$.  We begin by determining the intersection of the categories $\OCM(R) $ and $\SCM(R) $ when $R$ is not Gorenstein.

\begin{prop}\label{SCM and OCM}
If $R$ is not Gorenstein, then:\\
\t{(1)} If $X\in\CM(R)$ such that $\Ext^i_R(X,R)=0$ for $i=1,2$ then $X$ is free.\\
\t{(2)} $\SCM(R)\cap \OCM(R)=\add R$.
\end{prop}
\begin{proof}
(1) Since $\Ext^1_R(X,R)=0$, by Theorem~\ref{characterization of SCM} we know that $X\in\SCM(R)$ and so $\Omega X\cong X^*$.  Further $0=\Ext^2_R(X,R)=\Ext^1_R(\Omega X,R)=\Ext^1_R(X^*,R)$ and so $X^*\in\SCM(R)$. Now applying Theorem \ref{characterization of SCM}(7) to both $X$ and $X^*$ there exist  short exact sequences
\begin{eqnarray}\label{Xses}
&&0\to X^*\to P\to X\to 0\\
&&0\to X\to Q\to X^*\to 0\label{Xstarses}
\end{eqnarray}
with $P,Q\in\add R$.  We want to prove that $\Ext^t_R(X,R)=\Ext^t_R(X^*,R)=0$ for all $t\geq 1$ so since we know this holds for $t=1$, inductively suppose that it holds for $t-1$.  Then by (\ref{Xses}) we know
\[
\Ext_R^t(X,R)\cong \Ext_R^{t-1}(X^*,R)=0
\]
and by (\ref{Xstarses}) we know
\[
\Ext_R^t(X^*,R)\cong \Ext_R^{t-1}(X,R)=0.
\]
Thus by induction it follows that $\Ext_R^t(X\oplus X^*,R)=0$ for all $t\geq 1$ and hence by definition $X$ is a totally reflexive module (see \cite{Takahashi1,Takahashi2}).
But since $R$ has only finitely many indecomposable special CM modules it has in particular only finitely many indecomposable totally reflexive modules.  If $X$ is non-free then by \cite{Takahashi1} (see also \cite[4.3]{Takahashi2}) it follows that $R$ is Gorenstein.  Hence $X$ must be free.\\
(2) It follows from (1) that if $X$ and $X^*$ are special, then $X^*$ is free.
\end{proof}

\begin{remark}
An easy conclusion of Proposition \ref{SCM and OCM} and Theorem \ref{characterization of SCM}(7) is that if $M$ is a non-free special CM module, then $\Omega^2M\cong\Omega(M^*)$ is never isomorphic to $M$. This is in contrast to the Gorenstein case in which $\Omega^2M$ is always isomorphic to $M$ \cite{Eisenbud}.
\end{remark}

We already know that the dual of the first syzygy of any CM module is special; the following result gives us precise information about the decomposition into indecomposables.
\begin{thm}\label{syzygy}
Let $M$ be a CM $R$-module.  Then $\Omega M\cong \bigoplus_{i\in I} (\Omega M_i)^{\oplus c_1(M)\cdot E_i}$.
\end{thm}
\begin{proof}
Since $M$ is CM, by Artin-Verdier \cite[1.2]{ArtVer} we have the following exact sequence
\[
0\to \s{O}^{\oplus r}\to \c{M}\to \s{O}_{D} \to 0
\]
where $r$ is the rank of $M$ and $D$ represents the chern class of $\c{M}$, i.e. $c_1(M)\cdot E_i=D\cdot E_i$ for all exceptional curves $E_i$.  Let $\m$ denote the unique maximal ideal of $R$, then the minimal number of generators of the $R$-module $H^0(\s{O}_D)$ is, by Nakayama's Lemma, the dimension of the vector space $H^0(\s{O}_D)/\m H^0(\s{O}_D)=H^0(\s{O}_D/\m\s{O}_D)=H^0(\s{O}_D\otimes\s{O}_{Z_f})$, which is precisely $Z_f\cdot D$.  Hence choosing such a set of generators yields an exact sequence
\[
0\to \s{K}\to \s{O}^{\oplus Z_f\cdot D}\to \s{O}_{D} \to 0.
\]

Exactly as in the proof of \cite[1.2(a)]{Wunram_generalpaper} (Wunram considered the case when $M$ is indecomposable, but his proof works in this more general setting) $\s{K}^*$ is a full sheaf of rank $ Z_f\cdot D$ satisfying $H^1(\s{K})=0$.  Thus there exists some special CM $R$-module $N$ such that $\s{K}^*=\c{N}$, with $c_1(N)\cdot E_i=D\cdot E_i=c_1(M)\cdot E_i$ for all $i$.

The decomposition of $N$ into indecomposable special CM $R$-modules
\[
N=R^{\oplus s}\oplus(\bigoplus_{i\in I} M_i^{\oplus b_i})
\]
gives a corresponding decomposition of $\c{N}$.  Using Theorem~\ref{Wurnam_main_result} we have 
\[
b_i=c_1(N)\cdot E_i=c_1(M)\cdot E_i
\]
for each $i\in I$. Thus we have
\[
\c{N}=\s{O}^{\oplus s}\oplus(\bigoplus_{i\in I} \c{M}_i^{\oplus c_1(M)\cdot E_i})
\]
for some $s\in\mathbb{N}$.  The fact that $s=0$ follows by running the argument in \cite[1.2(a)]{Wunram_generalpaper}, or alternatively by using \cite[3.5.3]{VdB}.  
Hence we have a short exact sequence
\[
0\to \bigoplus_{i\in I}  (\c{M}_i^*)^{\oplus c_1(M)\cdot E_i}\to \s{O}^{\oplus Z_f\cdot D}\to \s{O}_{D} \to 0
\]
from which taking the appropriate pullback gives us a diagram
\[
{\SelectTips{cm}{10}
\xy0;/r.5pc/:
(30,12)*+{0}="m2",(45,12)*+{0}="m3",
(30,6)*+{\bigoplus (\c{M}_i^*)^{\oplus c_1(M)\cdot E_i}}="a2",(45,6)*+{\bigoplus
(\c{M}_i^*)^{\oplus c_1(M)\cdot E_i}}="a3",
(10,0)*+{0}="0",(20,0)*+{\s{O}^{\oplus
r}}="1",(30,0)*+{\s{E}}="2",(45,0)*+{\s{O}^{\oplus Z_f\cdot
D}}="3",(55,0)*+{0}="4",
(10,-6)*+{0}="b0",(20,-6)*+{\s{O}^{\oplus
r}}="b1",(30,-6)*+{\c{M}}="b2",(45,-6)*+{\s{O}_D}="b3",(55,-6)*+{0}="b4",
,(30,-12)*+{0}="c2",(45,-12)*+{0}="c3"
\ar"0";"1"
\ar"1";"2"
\ar"2";"3"
\ar"3";"4"
\ar"b0";"b1"
\ar"b1";"b2"
\ar"b2";"b3"
\ar"b3";"b4"
\ar@{=}"1";"b1"
\ar@{=}"a2";"a3"
\ar"m2";"a2"
\ar"m3";"a3"
\ar"a2";"2"
\ar"a3";"3"
\ar"2";"b2"
\ar"3";"b3"
\ar"b2";"c2"
\ar"b3";"c3"
\endxy}
\]
Since $\pi:\w{X}\rightarrow \Spec R$ is a resolution of rational singularities $H^1(\s{O})=\Ext_{\w{X}}^1(\s{O},\s{O})=0$ and so the middle horizontal sequence splits, giving $\s{E}=\s{O}^{\oplus r+Z_f\cdot D}$.  Now since $H^1( \bigoplus_{i\in I} (\c{M}_i^*)^{\oplus c_1(M)\cdot E_i})=0$ we may push down the middle vertical sequence to obtain the short exact sequence
\[
0\to\bigoplus_{i\in I} M_i^{*\oplus c_1(M)\cdot E_i} \to R^{\oplus r+Z_f\cdot c_1(M)} \to M \to 0.
\]
Since by Theorem~\ref{characterization of SCM} $\Omega M_i\cong M_i^*$, the result follows.
\end{proof}
\begin{remark}
The above theorem gives us a global combinatorial method for computing chern classes of full sheaves in the cases of quotient singularities that doesn't resort to calculating with local co-ordinates on the minimal resolution, since the syzygy of any CM module can be easily calculated by using a counting argument on the AR(=McKay) quiver.  For details see \cite[4.9, 4.10]{Iyama_Wemyss_specials}.
\end{remark}
\begin{remark}
Since the above first syzygy contains no free summands it follows that any CM module $M$ is minimally generated by $\t{rk}M+Z_f\cdot c_1(M)$ elements.  This gives a new proof of \cite[2.1]{Wunram_generalpaper}.
\end{remark}

The following observation will be used in the next section.

\begin{cor}\label{chern example}
\t{(1)} We have $\Omega\omega\cong\bigoplus_{i\in I}(\Omega M_i)^{\oplus -E_i^2-2}$.\\
\t{(2)} If $R$ is not Gorenstein and $\omega$ is a special CM $R$-module, then the exceptional curve corresponding to $\omega$ is a $(-3)$-curve and all other exceptional curves are $(-2)$-curves.
\end{cor}

\begin{proof}
(1) By Theorem~\ref{syzygy} $\Omega\omega\cong \bigoplus_{i\in I}(\Omega M_i)^{\oplus  E_i\cdot K_{\w{X}}}$.  Further the adjunction formula states that $-2=(K_{\w{X}}+E_i)\cdot E_i$ and so $K_{\w{X}}\cdot E_i=-E_i^2-2$.\\
(2) Immediate from (1).
\end{proof}

\section{A Frobenius Structure on $\SCM(R)$}
In this section we endow the category $\SCM(R)$ with a Frobenius structure and thus produce a triangulated category $\u{\u{\SCM}}(R)$.  
We say that an extension closed subcategory $\BB$ of an abelian category $\AA$ is an \emph{exact category}.
(This is slightly stronger than the formal definition by Quillen \cite{Qui73}. See also \cite[Appendix A]{Kel90}.)
For example $\CM(R)$ is an exact category.

We start with the following easy observation.

\begin{lemma}
\t{(1)} $\SCM(R)$ is an extension closed subcategory of $\CM(R)$.\\
\t{(2)} $\SCM(R)$ forms an exact category.
\end{lemma}

\begin{proof}
(1) is an immediate consequence of Theorem \ref{characterization of SCM}(5),
and (2) is a consequence of (1).
\end{proof}

Let us recall the definition of Frobenius categories \cite{Hel60,Happel}.
We say that an object $X\in\BB$ is \emph{relatively projective} (respectively, \emph{relatively injective}) if
\[\Ext^1_{\AA}(X,\BB)=0\ \ \ (\mbox{respectively, } \Ext^1_{\AA}(\BB,X)=0).\]
We say that $\BB$ has \emph{enough relatively projective objects} (respectively, \emph{injective}) if for any $X\in\BB$, there exists an exact sequence
\[0\to Z\to Y\to X\to0\ \ \ (\mbox{respectively, } 0\to X\to Y\to Z\to0)\]
in $\AA$ such that $Y\in\BB$ is relatively projective (respectively, injective) and $Z\in\BB$.
We say that $\BB$ is \emph{Frobenius} if it has enough relatively projective and enough relatively injective objects, and further the relatively projective and the relatively injective objects coincide.
When $\BB$ is a Frobenius category with the subcategory $\mathcal{P}$ of relatively projective objects, the factor category
\[\u{\u{\BB}}:=\BB/[\mathcal{P}]\]
is called the \emph{stable category} of $\BB$.
The reason why we use the notation $\u{\u{\BB}}$ is to distinguish the stable category $\u{\u{\SCM}}(R)$ (which we will study) from the full subcategory $\u{\SCM}(R)$ of $\u{\CM}(R)$.

Our main result in this section is the following.

\begin{thm}\label{frobenius structure}
\t{(1)} $\SCM(R)$ is a Frobenius category.\\
\t{(2)} The stable category $\u{\u{\SCM}}(R)$ is a triangulated category.
\end{thm}

Let us recall the definition of functorially finite subcategories introduced by Auslander-Smal\o\, \cite{AS}.
Let $\BB$ be an additive category and $\CC$ a full subcategory of $\BB$.
We say that a subcategory $\CC$ of an additive category $\BB$ is \emph{contravariantly finite} (respectively, \emph{covariantly finite})  if for any $X\in\BB$,
there exists a morphism $f:Y\to X$ (respectively, $f:X\to Y$) with $Y\in\CC$ such that
\[
\Hom_{\BB}(\CC,Y)\to\Hom_{\BB}(\CC,X)\ \ \ (\mbox{respectively,} \Hom_{\BB}(Y,\CC)\to\Hom_{\BB}(X,\CC))
\]
is surjective.  We say that $\CC$ is a \emph{functorially finite} subcategory of $\BB$ if it is both contravariantly and covariantly finite.

We need the following rather general observation.

\begin{prop}
Let $\BB$ be a Krull-Schmidt exact category with enough relatively injective (respectively, projective) objects, and $\CC$ a contravariantly (respectively, covariantly) finite extension closed subcategory of $\BB$. Then $\CC$ is an exact category with enough relatively injective (respectively, projective) objects.
\end{prop}

\begin{proof}
We prove the statement regarding relatively injective objects; the proof for relatively projective objects is similar.  It is clear that $\CC$ is also an exact category. Let $X$ be in $\CC$ and
take an exact sequence $0 \to X \to I \to X' \to 0$ with $I$ relatively injective in $\BB$. Then we have an exact sequence of functors $\Hom_{\BB}(- ,X') \to \Ext ^1_{\AA}(- ,X)\to 0$. Since $\CC$ is Krull-Schmidt and contravariantly finite in $\BB$, we can take a projective cover
$\phi:\Hom_{\CC}(- ,Y)\to{\Ext}^1_{\AA}(- ,X)|_{\CC}\to0$ of $\CC$-modules (for the definition of $\CC$-modules see for example \cite{Y}).
This is induced by an exact sequence $0\to X\to Z\stackrel{}{\to}
Y\to0$ with terms in $\CC$.

We will show that $Z$ is relatively injective. Take any exact sequence $0\to
Z\stackrel{}{\to}Z'\to Z''\to0$ with terms in $\CC$. We will
show that this splits. Consider the following exact commutative diagram:
\begin{equation}\label{commutative diagram}
\begin{array}{c}
{\SelectTips{cm}{10}
\xy0;/r.28pc/:
(20,8)*+{0}="a2",(30,8)*+{0}="a3",
(0,0)*+{0}="0",(10,0)*+{X}="1",(20,0)*+{Z}="2",(30,0)*+{Y}="3", (40,0)*+{0}="4" , 
(0,-8)*+{0}="b0",(10,-8)*+{X}="b1",(20,-8)*+{Z^\prime}="b2",(30,-8)*+{Y^\prime}="b3",(40,-8)*+{0}="b4" ,
(20,-16)*+{Z^{\prime\prime}}="c2",(30,-16)*+{Z^{\prime\prime}}="c3",
(20,-24)*+{0}="d2",(30,-24)*+{0}="d3",
\ar"a2";"2"
\ar"a3";"3"
\ar"0";"1"
\ar"1";"2"
\ar"2";"3"
\ar"3";"4"
\ar@{=}"1";"b1"
\ar"2";"b2"
\ar^{a}"3";"b3"
\ar"b0";"b1"
\ar"b1";"b2"
\ar"b2";"b3"
\ar"b3";"b4"
\ar"b2";"c2"
\ar"b3";"c3"
\ar@{=}"c2";"c3"
\ar"c2";"d2"
\ar"c3";"d3"
\endxy}
\end{array}
\end{equation}
Then $Y'\in\CC$, and we have the commutative diagram
\begin{equation}\label{commutative diagram2}
\begin{array}{c}
{\SelectTips{cm}{10}
\xy0;/r.39pc/:
(5,0)*+{0}="0",(15,0)*+{\Hom_{\CC}(-,X)}="1",(30,0)*+{\Hom_{\CC}(-,Z)}="2",(45,0)*+{\Hom_{\CC}(-,Y)}="3",(60,0)*+{\Ext^1_{\AA}(-,X)|_{\CC}}="4", (70,0)*+{0}="5" , (5,-6)*+{0}="b0",(15,-6)*+{\Hom_{\CC}(-,X)}="b1",(30,-6)*+{\Hom_{\CC}(-,Z^\prime)}="b2",(45,-6)*+{\Hom_{\CC}(-,Y^\prime)}="b3",(60,-6)*+{\Ext^1_{\AA}(-,X)|_{\CC}}="b4"
\ar"0";"1"
\ar"1";"2"
\ar"2";"3"
\ar"3";"4"^{\phi}
\ar"4";"5"
\ar@{=}"1";"b1"
\ar"2";"b2"
\ar^{\cdot a}"3";"b3"
\ar@{=}"4";"b4"
\ar"b0";"b1"
\ar"b1";"b2"
\ar"b2";"b3"
\ar"b3";"b4"
\endxy}
\end{array}
\end{equation}
of exact sequences of $\CC$-modules. Since $\phi$ is a projective cover, we have that
$(\cdot a)$ is a split monomorphism. Thus $a$ is a split monomorphism.
We see that the sequence $0 \to \Ext^1(Z'',Z) \to \Ext^1(Z'',Y)$ is
exact by evaluating the upper sequence in \eqref{commutative diagram2}
at $Z''$. Under this map the middle vertical exact sequence in \eqref{commutative diagram} gets sent to the right vertical exact sequence in \eqref{commutative diagram}, so since this splits it follows that the middle
vertical sequence in \eqref{commutative diagram} splits. Hence $Z$ is
relatively injective, and consequently $\CC$ has enough relatively injective objects.
\end{proof}

In particular, since $\CM(R)$ has enough relatively projective and injective objects and further $\SCM(R)$ is a functorially finite subcategory of $\CM(R)$,
we conclude that $\SCM(R)$ has enough relatively projective and injective objects.

We need the following observation:
\begin{lemma}
For any $X,Y\in\SCM(R)$ we have $\Ext^1_R(X,Y)\cong \Ext^1_R(Y,X)$.
\end{lemma}
\begin{proof}
It is standard that $\u{\Hom}_R(\Omega X,Y)\cong\Ext^1_R(X,Y)$.  Hence since $\Omega X\cong X^*$ and $\Omega Y\cong Y^*$  by Theorem~\ref{characterization of SCM}, we have
\[
\Ext^1_R(X,Y)\cong \u{\Hom}_R(X^*,Y)\cong \u{\Hom}_R(Y^*,X)\cong\Ext^1_R(Y,X).
\]
where the middle isomorphism is given by the duality $(-)^*:\underline{\CM}(R)\to\underline{\CM}(R)$ induced from the duality $(-)^*:\CM(R)\to \CM(R)$.
 \end{proof}

Thus in $\SCM(R)$ the relatively projective objects and the relatively injective objects coincide and so
consequently $\SCM(R)$ is a Frobenius category and thus $\u{\u{\SCM}}(R)$ is triangulated \cite{Happel}.
This completes the proof of Theorem \ref{frobenius structure}.
\qed

\medskip
The next result gives a precise description of the relatively projective objects:

\begin{thm}\label{-2 curve}
Let $i\in I$. Then $M_i$ is relatively projective in $\SCM(R)$ if and only if $E_i$ is not a $(-2)$-curve.
\end{thm}

We divide the proof into Lemmas \ref{-2} and \ref{non-2}.  For the first we require the following well-known observation.
\begin{lemma}\label{3.5new}
Let $R$ be a noetherian ring and $M\in \mod R$.
If $R\in \add M$, then the functor
\[ 
\Hom_R(M,-): \mod R\to \mod \End_R(M)
\]
is fully faithful, restricting to an equivalence
\[ 
\add M\to \proj \End_R(M).
\] 
\end{lemma}

\begin{lemma}\label{-2}
If $E_i$ is a $(-2)$--curve, then $M_i$ is not relatively projective in $\SCM(R)$ and further
$\Ext^1_R(M_i,M_i)\neq0$.
\end{lemma}
\begin{proof}
Denote $M:=R\oplus\bigoplus_{j\in I}M_j$ and let $A:=\End_R(M)$.   We show that $\Ext^1_R(M_i,M_i)\neq 0$ by using the fact that $A$ is derived equivalent to the minimal resolution (Theorem~\ref{derivedequiv}).  For all $j\in I$ denote $S_j$ to be the simple at the vertex corresponding to $M_j$ in the quiver of $A$ (i.e.\ $S_j$ is the top of $\Hom_R(M,M_j)$) and let $S_\star$ be the simple corresponding to $R$ in the quiver of $A$ (i.e.\ $S_\star$ is the top of $\Hom_R(M,R)$).  Then by inspecting the proof of \cite[Thm 3.1]{Wemyss_GL2} we see that $\Ext^3_A(S_i,S_j)=0$ for all $j\in I$ and further $\Ext^3_A(S_i,S_\star)=-E_i^2-2$.  Thus since $M_i$ corresponds to a $(-2)$ curve, $\Ext^3_A(S_i,-)=0$ against all simple $A$ modules and so $\pd_A S_i=2$.  Consider now the minimal projective resolution of the $A$-module $S_i$, which by Lemma~\ref{3.5new} has the form 
\[
0\to \Hom_R(M,T) \to \Hom_R(M,Y) \to\Hom_R(M,M_i)  \to S_i\to 0.
\]
We know that this comes from a non-split exact sequence
\[
0\to T\to Y\to M_i\to 0
\]
with $T,Y\in\add M=\SCM(R)$.  Now by \cite[Thm 3.1]{Wemyss_GL2} we know that $\Ext^2_A(S_i, S_\star)=((Z_K-Z_f)\cdot E_i)_{-}=0$ (since $Z_K\cdot E_i=0$ and so $(Z_K-Z_f)\cdot E_i\geq 0$), $\Ext^2_A(S_i, S_j)=0$ if $i\neq j$ and further $\Ext^2_A(S_i, S_i)=-E_i^2-1=1$ since $E_i$ is a $(-2)$-curve. Consequently $\Hom_R(M,T)\cong\Hom_R(M,M_i)$ and so $T\cong M_i$ by Lemma~\ref{3.5new}.  Hence $\Ext^1_R(M_i,M_i)\neq 0$, as required.
\end{proof}

\begin{lemma}\label{non-2}
If $E_i$ is not a $(-2)$-curve, then $M_i$ is relatively projective in $\SCM(R)$ and $\Ext^1_R(M_i,M_i)=0$.
\end{lemma}
\begin{proof}
Firstly note that for all $X,Y\in\CM(R)$, if $\Ext_R^1(X,Y)=0$ then necessarily $\Ext^1_R(\tau^{-1}\Omega^{-1}Y,X)=0$.  To see this, just take the short exact sequence $0\to Y\to I\to \Omega^{-1}Y \to 0$ with $I\in\add\omega$ and apply $\Hom_R(X,-)$ to get
\[
0\to \Hom_R(X,Y)\to \Hom_R(X,I)\to \Hom_R(X,\Omega^{-1}Y) \to \Ext^1_R(X,Y)=0.
\]
Consequently every map from $X$ to $\Omega^{-1}Y$ factors through an injective object and hence by AR  duality $0=D\overline{\Hom}_R(X,\Omega^{-1}Y)=\Ext^1_R(\tau^{-1}\Omega^{-1}Y,X)$.

Now if $X\in\SCM(R)$ then $\Ext^1_R(X,R)=0$ and so applying the above with $Y=R$ we get $\Ext^1_R(\tau^{-1}\Omega^{-1}R,X)=0$.  But $\tau^{-1}\Omega^{-1}R=\Hom_R(\Omega^{-1}R,\omega)^*=(\Omega \omega)^*$ and so this shows that $\Ext^1_R((\Omega \omega)^*,-)=0$ on $\SCM(R)$, hence $(\Omega\omega)^*$ is relatively projective.  But now by Corollary~\ref{chern example},
we know that $(\Omega\omega)^*$ has as summands all the indecomposable special CM modules corresponding to non-$(-2)$ curves and thus all of them are relatively projective. 
\end{proof}

This completes the proof of Theorem \ref{-2 curve}.
\qed

\medskip
We now show the following existence theorem of almost split sequences in $\SCM(R)$.  The theory of almost split sequences in subcategories was first developed by Auslander and Smal\o\, \cite{AS} for finite dimensional algebras; here our algebras are not finite dimensional, but the proofs are rather similar.

Below we denote by $J_{\CM(R)}$ the \emph{Jacobson radical} of the category $\CM(R)$ (e.g. \cite{ARS}), so $J_{\CM(R)}(X,Y)$ consists of non-isomorphic morphisms $X\to Y$ for any indecomposable CM $R$-modules $X$ and $Y$.
\begin{prop}\label{recon_relation}
Let $i\in I$. Then $M_i$ is not relatively projective if and only if there exists an exact sequence
\[0\to M_i\xrightarrow{g} Y\xrightarrow{f} M_i\to0\]
such that the sequences
\begin{eqnarray*}
&0\to\Hom_R(-,M_i)\xrightarrow{\cdot g}\Hom_R(-,Y)\xrightarrow{\cdot f}J_{\CM(R)}(-,M_i)\to0,&\\
&0\to\Hom_R(M_i,-)\xrightarrow{f\cdot }\Hom_R(Y,-)\xrightarrow{g\cdot }J_{\CM(R)}(M_i,-)\to0.&
\end{eqnarray*}
are exact on $\SCM(R)$.
\end{prop}
\begin{proof}
Suppose $M_i$ not relatively projective.  We firstly show that there exists an almost split sequence $0\to Z\to Y\to M_i\to 0$ in $\SCM(R)$.  Since there are only finitely many indecomposable objects in $\SCM(R)$, certainly there exists an exact sequence 
\[
0\to Z\stackrel{g}\to Y\stackrel{f}\to M_i\to 0
\]
with $Y\in\SCM(R)$ and $f$ a minimal right almost split map in $\SCM(R)$.  We claim that $Z\in\SCM(R)$ and further $g$ is a minimal left almost split map.

Since $M_i$ is not relatively projective, there exists an exact sequence
\[
0\to Z^\prime\to Y^\prime\to M_i\to 0
\]
with $Z^\prime\in\SCM(R)$.  Since $f$ is right almost split we have a commutative diagram 
\[
{\SelectTips{cm}{10}
\xy0;/r.28pc/:
(0,0)*+{0}="0",(10,0)*+{Z^\prime}="1",(20,0)*+{Y^\prime}="2",(30,0)*+{M_i}="3", (40,0)*+{0}="4" , 
(0,-8)*+{0}="b0",(10,-8)*+{Z}="b1",(20,-8)*+{Y}="b2",(30,-8)*+{M_i}="b3",(40,-8)*+{0}="b4" ,
\ar"0";"1"
\ar"1";"2"
\ar"2";"3"
\ar"3";"4"
\ar"1";"b1"
\ar"2";"b2"
\ar@{=}"3";"b3"
\ar"b0";"b1"
\ar"b1";"b2"
\ar^f"b2";"b3"
\ar"b3";"b4"
\endxy}
\]
and so taking the mapping cone gives the short exact sequence
\[
0\to Z^\prime\to Z\oplus Y^\prime\to Y\to 0.
\]
Since $\SCM(R)$ is closed under extensions we conclude that $Z\in\SCM(R)$.  The fact $g$ is a minimal left almost split map is now routine (see e.g. \cite[2.14]{Y}).

To finish the proof we must show that $Z\cong M_i$.  But as in the proof of Lemma~\ref{-2}, the above gives a minimal projective resolution
\[
0\to\Hom_R(M,Z)\to\Hom_R(M,Y)\to\Hom_R(M,M_i)\to S_i\to 0
\]
of the simple $A$-module $S_i$.  By \cite[Thm 3.1]{Wemyss_GL2} we know that $\Ext^2_A(S_i,S)=0$ for any simple $A$-module $S\neq S_i$.  Thus $\Hom_R(M,Z)\cong\Hom_R(M,M_i)$ and so by Lemma~\ref{3.5new} $Z\cong M_i$, as required.
\end{proof}

The following property of `Auslander algebras' of triangulated categories is useful.

\begin{prop}\label{triangulated Auslander algebra}
Let $\TT$ be a $\Hom$-finite $k$-linear triangulated category $\TT$ with an additive generator $M$.
Then $B:=\End_{\TT}(M)$ is a self-injective $k$-algebra.
\end{prop}

\begin{proof}
For any $X\in\mod B$, we can take a projective resolution
\[
\Hom_{\TT}(M,M_1)\xrightarrow{\cdot f}\Hom_{\TT}(M,M_0)\to X\to0.
\]
Take a triangle $M_2\xrightarrow{g} M_1\xrightarrow{f} M_0\to M_2[1]$,
then we continue a projective resolution
\[
\Hom_{\TT}(M,M_2)\xrightarrow{\cdot g}\Hom_{\TT}(M,M_1)\xrightarrow{\cdot f}\Hom_{\TT}(M,M_0)\to X\to0.
\]
Applying $\Hom_B(-,B)$ gives the commutative diagram
\[
{\SelectTips{cm}{10}
\xy0;/r.4pc/:
(0,0)*+{\Hom_B(\Hom_{\TT}(M,M_0),B)}="1",(26,0)*+{\Hom_B(\Hom_{\TT}(M,M_1),B)}="2",(52,0)*+{\Hom_B(\Hom_{\TT}(M,M_2),B)}="3",
(0,-6)*+{\Hom_{\TT}(M_0,M)}="b1",(26,-6)*+{\Hom_{\TT}(M_1,M)}="b2",(52,-6)*+{\Hom_{\TT}(M_2,M)}="b3"
\ar"1";"2"
\ar"2";"3"
\ar"1";"b1"
\ar"2";"b2"
\ar"3";"b3"
\ar^{f\cdot }"b1";"b2"
\ar^{g\cdot }"b2";"b3"
\endxy}
\]
where all vertical maps are isomorphisms.  Since the lower sequence is exact (by properties of triangles), so is the top.  Hence $\Ext^1_B(X,B)=0$ and so $B$ is self-injective.
\end{proof}

We deduce the following results on our triangulated category $\u{\u{\SCM}}(R)$ and the stable reconstruction algebra $\End_{\u{\u{\SCM}}(R)}(\bigoplus_{i\in I} M_i)$.

\begin{cor}
\t{(1)} The AR quiver of the category $\u{\u{\SCM}}(R)$ is a disjoint union of the double of Dynkin diagrams, corresponding to the subconfigurations of $(-2)$-curves in the minimal resolution.\\
\t{(2)} The algebra $\End_{\u{\u{\SCM}}(R)}(\bigoplus_{i\in I} M_i)$ is a factor algebra of the reconstruction algebra $\End_R(R\oplus\bigoplus_{i\in I} M_i)$ by the ideal generated by idempotents corresponding to $R$ and the non-$(-2)$-curves. \\
\t{(3)} The algebra $\End_{\u{\u{\SCM}}(R)}(\bigoplus_{i\in I} M_i)$ is self-injective, and the quiver is a disjoint union of the double of Dynkin diagrams.
\end{cor}
\begin{proof}
(1) A subtree of a rational tree is rational (see e.g. \cite[3.2]{TT}), thus the remaining ($-2$)-configurations are all Dynkin diagrams.  Alternatively, it is well-known that the AR quiver of a $\Hom$-finite $k$-linear triangulated category of finite type is a disjoint union of Dynkin diagrams \cite{XZ}.\\
(2) This is clear.\\
(3) Immediate from  Lemma \ref{triangulated Auslander algebra} since $\bigoplus_{i\in I} M_i$ is an additive generator of the triangulated category $\u{\u{\SCM}}(R)$.
\end{proof}

The following are examples which illustrate the above results.  Note that the quiver of the reconstruction algebra follows easily from combinatorics on the dual graph, see \cite{Wemyss_GL2} for details.
\begin{example}
\[
\begin{tikzpicture}
\node at (-3.75,2) {\rm Dual Graph};
\node at (0,2) {\rm Reconstruction Algebra};
\node at (4.25,1.95) {\rm AR quiver of $\u{\u{\SCM R}}$};
\node at (-3.75,0.3) 
{\begin{tikzpicture}[xscale=0.75,yscale=0.75]
 \node (0) at (0,0) [vertex] {};
 \node (1) at (1,0) [vertex] {};
 \node (1b) at (1,1) [vertex] {};
 \node (2) at (2,0) [vertex] {};
  \node (3) at (3,0) [vertex] {};
 \node (0a) at (-0.1,-0.3) {$\scriptstyle -2$};
 \node (1a) at (0.9,-0.3) {$\scriptstyle -2$};
 \node (1ba) at (0.55,1) {$\scriptstyle -2$};
 \node (2a) at (1.9,-0.3) {$\scriptstyle -2$};
  \node (3a) at (2.9,-0.3) {$\scriptstyle -4$};
 \draw [-] (0) -- (1);
\draw [-] (1) -- (2);
\draw [-] (2) -- (3);
\draw [-] (1) -- (1b);
\end{tikzpicture}};
\node at (0,0) {\begin{tikzpicture}[xscale=1,yscale=1,bend angle=12.5, looseness=1]
 \node (1) at (1,0) [vertex] {};
 \node (1b) at (2,1) [vertex] {};
 \node (2) at (2,0)  [vertex] {};
 \node (3) at (3,0)  [vertex] {};
 \node (4) at (4,0)  [vertex] {};
\node (R) at (3,-1) {$\scriptstyle \star$};
\draw [bend right,->] (1) to (2);
\draw [bend right,->] (2) to (1);
\draw [bend right,->] (2) to (3);
\draw [bend right,->] (3) to (2);
\draw [bend right,->] (3) to (4);
\draw [bend right,->] (4) to (3);
\draw [bend right,->] (2) to (1b);
\draw [bend right,->] (1b) to (2);
\draw [bend right,->] (3) to (R);
\draw [bend right,->] (R) to (3);
\draw [green,->] ($(4)+(-130:5pt)$) -- ($(R)+(50:4pt)$);
\draw [green,->] ($(4)+(-110:5pt)$) -- ($(R)+(30:4pt)$);
\end{tikzpicture}};
\node at (4.25,0) {\begin{tikzpicture}[xscale=1,yscale=1,bend angle=12.5, looseness=1]
 \node (1) at (1,0) [vertex] {};
 \node (1b) at (2,1) [vertex] {};
 \node (2) at (2,0)  [vertex] {};
 \node (3) at (3,0)  [vertex] {};
  \node (4) at (4,0)  {};
\node (R) at (3,-1) {};
\draw [bend right,->] (1) to (2);
\draw [bend right,->] (2) to (1);
\draw [bend right,->] (2) to (3);
\draw [bend right,->] (3) to (2);
\draw [bend right,->] (2) to (1b);
\draw [bend right,->] (1b) to (2);
\end{tikzpicture}};
\end{tikzpicture}
\]
\[
\begin{tikzpicture}
\node at (-3.75,0.3) 
{\begin{tikzpicture}[xscale=0.75,yscale=0.75]
 \node (0) at (0,0) [vertex] {};
 \node (1) at (1,0) [vertex] {};
 \node (1b) at (1,1) [vertex] {};
 \node (2) at (2,0) [vertex] {};
  \node (3) at (3,0) [vertex] {};
 \node (0a) at (-0.1,-0.3) {$\scriptstyle -2$};
 \node (1a) at (0.9,-0.3) {$\scriptstyle -2$};
 \node (1ba) at (0.55,1) {$\scriptstyle -2$};
 \node (2a) at (1.9,-0.3) {$\scriptstyle -4$};
  \node (3a) at (2.9,-0.3) {$\scriptstyle -2$};
 \draw [-] (0) -- (1);
\draw [-] (1) -- (2);
\draw [-] (2) -- (3);
\draw [-] (1) -- (1b);
\end{tikzpicture}};
\node at (0,0) {\begin{tikzpicture}[xscale=1,yscale=1,bend angle=12.5, looseness=1]
 \node (1) at (1,0) [vertex] {};
 \node (1b) at (2,1) [vertex] {};
 \node (2) at (2,0) [vertex] {};
 \node (3) at (3,0) [vertex] {};
\node (4) at (4,0) [vertex] {};
\node (R) at (2,-1) {$\scriptstyle \star$};
\draw [bend right,->] (1) to (2);
\draw [bend right,->] (2) to (1);
\draw [bend right,->] (2) to (3);
\draw [bend right,->] (3) to (2);
\draw [bend right,->] (3) to (4);
\draw [bend right,->] (4) to (3);
\draw [bend right,->] (2) to (1b);
\draw [bend right,->] (1b) to (2);
\draw [bend right,->] (2) to (R);
\draw [bend right,->] (R) to (2);
\draw [green,->] ($(3)+(-120:5pt)$) -- ($(R)+(70:4pt)$);
\draw [bend right=3,->] ($(R)+(25:4pt)$) to ($(4)+(-100:4pt)$);
\draw [bend right=3,->] ($(4)+(-125:4pt)$) to ($(R)+(45:4pt)$);
\end{tikzpicture}};
\node at (4.25,0) {\begin{tikzpicture}[xscale=1,yscale=1,bend angle=12.5, looseness=1]
 \node (1) at (1,0) [vertex] {};
 \node (1b) at (2,1) [vertex] {};
 \node (2) at (2,0)  [vertex] {};
 \node (3) at (3,0)  {};
  \node (4) at (4,0) [vertex] {};
\node (R) at (3,-1) {};
\draw [bend right,->] (1) to (2);
\draw [bend right,->] (2) to (1);
\draw [bend right,->] (2) to (1b);
\draw [bend right,->] (1b) to (2);
\end{tikzpicture}};
\end{tikzpicture}
\]
\[
\begin{tikzpicture}
\node at (-3.75,0.3) 
{\begin{tikzpicture}[xscale=0.75,yscale=0.75]
 \node (0) at (0,0) [vertex] {};
 \node (1) at (1,0) [vertex] {};
 \node (1b) at (1,1) [vertex] {};
 \node (2) at (2,0) [vertex] {};
  \node (3) at (3,0) [vertex] {};
 \node (0a) at (-0.1,-0.3) {$\scriptstyle -2$};
 \node (1a) at (0.9,-0.3) {$\scriptstyle -4$};
 \node (1ba) at (0.55,1) {$\scriptstyle -2$};
 \node (2a) at (1.9,-0.3) {$\scriptstyle -2$};
  \node (3a) at (2.9,-0.3) {$\scriptstyle -2$};
 \draw [-] (0) -- (1);
\draw [-] (1) -- (2);
\draw [-] (2) -- (3);
\draw [-] (1) -- (1b);
\end{tikzpicture}};
\node at (0,0) {\begin{tikzpicture}[xscale=1,yscale=1,bend angle=12.5, looseness=1]
 \node (1) at (1,0) [vertex] {};
 \node (1b) at (2,1)[vertex] {};
 \node (2) at (2,0) [vertex] {};
 \node (3) at (3,0) [vertex] {};
 \node (4) at (4,0) [vertex] {};
\node (R) at (2,-1) {$\scriptstyle \star$};
\draw [bend right,->] (1) to (2);
\draw [bend right,->] (2) to (1);
\draw [bend right,->] (2) to (3);
\draw [bend right,->] (3) to (2);
\draw [bend right,->] (3) to (4);
\draw [bend right,->] (4) to (3);
\draw [bend right,->] (2) to (1b);
\draw [bend right,->] (1b) to (2);
\draw [bend right=5,->] ($(1)+(-55:6pt)$) to ($(R)+(145:4pt)$);
\draw [bend right=5,->] (R) to (1);
\draw [bend right=3,->] ($(4)+(-150:5.5pt)$) to ($(R)+(20:4pt)$);
\draw [bend right=3,->] ($(R)+(5:5.5pt)$) to ($(4)+(-130:5.5pt)$);
\draw[->] (0.5,0) arc (180:116:2.5cm and 1.05cm); 
\draw (0.5,0) arc (-180:-117:2.5cm and 1.05cm);  
\draw (0.3,0) arc (180:112:2.5cm and 1.075cm); 
\draw[->] (0.3,0) arc (-180:-112:2.5cm and 1.075cm);  
\draw[->,green] ($(2)+(-90:4pt)$) -- ($(R)+(90:4pt)$);
\end{tikzpicture}};
\node at (4.25,0.025) {\begin{tikzpicture}[xscale=1,yscale=1,bend angle=12.5, looseness=1]
 \node (1) at (1,0) [vertex] {};
 \node (1b) at (2,1) [vertex] {};
 \node (2) at (2,0)  {};
 \node (3) at (3,0)  [vertex] {};
  \node (4) at (4,0) [vertex] {};
\node (R) at (3,-1) {};
\draw [bend right,->] (3) to (4);
\draw [bend right,->] (4) to (3);
\end{tikzpicture}};
\end{tikzpicture}
\]
\end{example}
\begin{remark}
Since the non $(-2)$-curves (and $R$) die in the quotient, often the AR quiver of $\u{\u{\SCM}}(R)$ has components.  In fact although the number of components is always finite, the number of possible components is arbitrarily large, as can be seen by constructing the following well-known rational tree: for any graph $\Gamma$ with vertices $E_i$ add self-intersection numbers as
\[
E_i^2:=\left\{\begin{array}{cl} -2&\t{if the number of neighbours of $E_i$ is one}\\ -\t{(number of neighbours of $E_i$)}&\t{else } \end{array}  \right.
\]
It is easy to check combinatorially that the above is a rational tree by using a result of Artin \cite[Thm.3]{Art} together with Riemann-Roch (see e.g. \cite[2.4]{TT}).  Thus the above example corresponds to the dual graph of some rational surface singularity and so in particular 
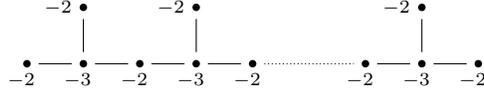
\[
\begin{tikzpicture}[xscale=0.75,yscale=0.75]
\node (0) at (0,0) [vertex] {};
\node (1) at (1,0) [vertex] {};
\node (1b) at (1,1) [vertex] {};
\node (2) at (2,0) [vertex] {};
\node (3) at (3,0) [vertex] {};
\node (3b) at (3,1) [vertex] {};
\node (4) at (4,0) [vertex] {};
\node (5) at (6,0) [vertex] {};
\node (6) at (7,0) [vertex] {};
\node (6b) at (7,1) [vertex] {};
\node (7) at (8,0) [vertex] {};
\node (0a) at (-0.1,-0.3) {$\scriptstyle -2$};
\node (1a) at (0.9,-0.3) {$\scriptstyle -3$};
\node (1ba) at (0.55,1) {$\scriptstyle -2$};
\node (2a) at (1.9,-0.3) {$\scriptstyle -2$};
\node (3a) at (2.9,-0.3) {$\scriptstyle -3$};
\node (2ba) at (2.55,1) {$\scriptstyle -2$};
\node (4a) at (3.9,-0.3) {$\scriptstyle -2$};
\node (5a) at (5.9,-0.3) {$\scriptstyle -2$};
\node (6a) at (6.9,-0.3) {$\scriptstyle -3$};
\node (6ba) at (6.55,1) {$\scriptstyle -2$};
\node (7a) at (7.9,-0.3) {$\scriptstyle -2$};
\draw [-] (0) -- (1);
\draw [-] (1) -- (2);
\draw [-] (2) -- (3);
\draw [-] (3) -- (4);
\draw[densely dotted] (4) -- (5);
\draw [-] (5) -- (6);
\draw [-] (6) -- (7);
\draw [-] (1) -- (1b);
\draw [-] (3) -- (3b);
\draw [-] (6) -- (6b);
\end{tikzpicture}
\]
(where in the region $\hdots$ we repeat the block on the right hand side) corresponds to some rational surface singularity.  On taking the quotient there are many components; increasing the size of the dual graph increases the number of such components.
\end{remark}

\begin{remark} 
Note that the above examples also illustrate that in many cases the category $\u{\u{\SCM}}(R)$ is equivalent to $\u{\CM}(R^\prime)$ for some Gorenstein ring $R^\prime$.
\end{remark}

We end by using our results to characterize those rational surfaces for which the category $\CM(R)$ contains an $n$-cluster tiliting object.  Recall that $M\in\CM(R)$ is called \emph{$n$-cluster tilting} (or \emph{maximal $(n-1)$-orthogonal}) for a positive integer $n$ \cite{I,KR} if
\begin{eqnarray*}
\add M&=&\{ X\in\CM(R) : \Ext^i_R(M,X)=0\ (0<i<n) \}\\
&=&\{ X\in\CM(R) : \Ext^i_R(X,M)=0\ (0<i<n)\}. 
\end{eqnarray*}
In this case, we have $\Ext^i_R(M,M)=0$ for any $0<i<n$ and $R\oplus\omega\in\add M$.
\begin{thm}
\t{(1)} $\CM(R)$ has a 1-cluster tilting object if and only if $R$ is a quotient singularity.\\
\t{(2)} $\CM(R)$ has a 2-cluster tilting object if and only if $R$ is regular or $R\cong k[[x,y]]^{\frac{1}{3}(1,1)}$ where $\frac{1}{3}(1,1)$ is the cyclic group of order 3 inside $\t{GL}(2,k)$ acting as $x\mapsto \varepsilon x$, $y\mapsto \varepsilon y$, where $\varepsilon$ is a cube root of unity.\\
\t{(3)} $\CM(R)$ has an $n$-cluster tilting object for some $n>2$ if and only if $R$ is regular.
\end{thm}
\begin{proof}
If $R$ is regular, then $\CM(R)=\add R$ and so $R$ is an $n$-cluster tilting object in $\CM(R)$ for any $n\geq1$. Hence we only need to consider the case when $R$ is not regular.\\
(1) By the Krull-Schmidt property, $\CM(R)$ has a 1-cluster tilting object if and only if $\CM(R)$ has finite type. By \cite{Auslander_rational} this is equivalent to $R$ being a quotient singularity.\\
(2) Let $M$ be a basic 2-cluster tilting object of $\CM(R)$. Since $R\in\add M$ and $\Ext^1_R(M,M)=0$, we have that $M$ is special.  Now since $\omega$ is a summand of $M$, this implies that $\omega$ is special.  Since $\Ext^1_R(M,M)=0$, by Lemma~\ref{-2} any non-free indecomposable summand of $M$ corresponds to a non-$(-2)$-curve.  In particular $R$ is not Gorenstein so by Corollary~\ref{chern example} the exceptional curve corresponding to $\omega$ is a $(-3)$-curve and all other exceptional curves are $(-2)$-curves.  This implies $M\cong R\oplus\omega$, so by Lemma~\ref{non-2} we have that $M$ is relatively projective in $\SCM(R)$.
Since $\Ext^1_R(M,\SCM(R))=0$, we have $\SCM(R)=\add M$.  Thus the minimal resolution of $\Spec R$ consists of only one $(-3)$-curve, so $R\cong k[[x,y]]^{\frac{1}{3}(1,1)}$ since quotient singularities are taut \cite[2.12]{Brieskorn}.  By inspection, in this case $R\oplus \omega$ is a 2-cluster tilting object.\\
(3) $\CM(R)$ does not have an $n$-cluster tilting object for $n>2$ by Proposition~\ref{SCM and OCM}(1) in the non-Gorenstein case, and by Lemma~\ref{-2} in the Gorenstein case.
\end{proof}

\end{document}